\documentclass{amsart}
\usepackage[a4paper, margin=1.2in]{geometry} % Ajusta los márgenes

\usepackage{cancel}
\usepackage{lineno}
\usepackage[backref=page]{hyperref}
\usepackage{amsthm}
\usepackage{amsmath,amssymb}
\usepackage{mathrsfs}
\usepackage[all]{xy}
\usepackage{calrsfs}
\usepackage{mathpazo}
\usepackage{xcolor}
\usepackage{comment}
\usepackage{eqnarray}
\usepackage[normalem]{ulem} %Tachar texto

\usepackage{graphicx}
\usepackage{tikz}
\usetikzlibrary{decorations.markings} %Para usar decorations en tikz.

\newcommand{\N}{\mathbb{N}}

\newcommand{\Aut}{\mathop{\mathrm{Aut}}}
\newcommand{\Deck}{\mathop{\mathrm{Deck}}}
\newcommand{\SHomeo}{\mathop{\mathrm{SHomeo}}}
\newcommand{\LHomeo}{\mathop{\mathrm{LHomeo}}}
\newcommand{\SMap}{\mathop{\mathrm{SMap}}}
\newcommand{\LMap}{\mathop{\mathrm{LMap}}}
\newcommand{\Map}{\mathop{\mathrm{Map}}}

\newcommand{\w}[1]{\widetilde{#1}}

\DeclareMathOperator{\Homeo}{Homeo}

\newtheorem{thm}{Theorem}[section]
\newtheorem{lemma}[thm]{Lemma}
\newtheorem{proposition}[thm]{Proposition}
\newtheorem{corollary}[thm]{Corollary}

\newtheorem{mainthm}{Theorem}
\newtheorem{maincor}[mainthm]{Corollary}

\theoremstyle{definition}
\newtheorem{definition}[thm]{Definition}

\theoremstyle{remark}
\newtheorem{remark}[thm]{Remark}

\numberwithin{equation}{section}
\usepackage{amsmath}
\begin{document}

\title[Birman--Hilden theory for big mapping class groups]{Birman--Hilden theory for big mapping class groups}

\author{Nestor Colin}
\address{Instituto de Matem\'aticas, Universidad Nacional Aut\'onoma de M\'exico,
Oaxaca de Ju\'arez, Oaxaca 68000, M\'exico}
\email{ncolin@im.umam.mx} 

\author{Rubén A. Hidalgo}
\address{Departamento de Matemática y Estadística, Universidad de La Frontera, Temuco 4811230, Chile.}
\email{ruben.hidalgo@ufrontera.cl}

\author{Rita Jiménez Rolland}
\address{Instituto de Matem\'aticas, Universidad Nacional Aut\'onoma de M\'exico,
Oaxaca de Ju\'arez, Oaxaca 68000, M\'exico}
\email{rita@im.unam.mx} 

\author{Israel Morales}
\address{Departamento de Matemática y Estadística, Universidad de La Frontera, Temuco 4811230, Chile.}
\email{israel.morales@ufrontera.cl} 

\author{Saúl Quispe}
\address{Departamento de Matemática y Estadística, Universidad de La Frontera, Temuco 4811230, Chile.}
\email{saul.quispe@ufrontera.cl}


% 57M12 Low-dimensional topology of special (e.g., branched) coverings
% 57M07 Topological methods in group theory
% 57M60 Group actions on manifolds and cell complexes in low dimensions
% 20F34 Fundamental groups and their automorphisms (group-theoretic aspects)
% 57K20: 2-dimensional topology (including mapping class groups of surfaces, Teichm¨uller theory, curve complexes, etc.)
% 55N35: Other homology theories in algebraic topology
% 55N20: Generalized (extraordinary) homology and cohomology theories in algebraic topology
% 20J05: Homological methods in group theory
% 57S05: Topological properties of groups of homeomorphisms or diffeomorphisms
% 20H10: Fuchsian groups and their generalizations
% 30F50: Klein Surfaces


\begin{abstract}
Let $S$ and $X$ be two connected topological surfaces without boundary, and assume that  $S$ is either of infinite type or has negative Euler characteristic. In this paper, we prove that if  $p:S\rightarrow X$ is  a {\it fully ramified} branched covering map, then $p$ satisfies the {\it Birman--Hilden property}. This generalizes a theorem of Winarski, and the known results in the literature, to the context of surfaces of infinite type and branched covering maps of infinite degree. As an application, we show that the mapping class group (respectively, the braid group on $k$-strands) of a non-orientable surface of infinite type can be realized as a subgroup of the mapping class group (respectively, the braid group on $2k$-strands) of its orientable double cover.
\end{abstract}

\maketitle

% \tableofcontents

%%%%%%%%%%%%%%%% INTRODUCTION %%%%%%%%%%%%%%%%%%%%%%%%%%
%%%%%%%%%%%%%%%%%%%%%%%%%%%%%%%%%%%%%%%%%%%%%%%%%%%%%%%
\section{Introduction}

The study of the interaction between mapping class groups through branched covers has its foundation in the work of Birman and Hilden in the 1970s \cite{BH71}, \cite{BH72}, \cite{BH73}. They established a criterion, now known as the \emph{Birman-Hilden property}, which provides a connection between certain subgroups of the mapping class groups of the total and base spaces of a branched cover. Specifically, a branched covering map $p: S \to X$  satisfies the \emph{Birman-Hilden property} if every fiber-preserving homeomorphism of $S$ that is isotopic to the identity admits an isotopy through fiber-preserving homeomorphisms.  In Section \ref{Sec:BH}, we review some equivalent formulations of this property for branched covers of topological surfaces, and its interpretation in terms of mapping class groups.

The purpose of this article is to extend the known results in the literature to the context of surfaces of infinite type and certain branched covering maps of infinite degree.  Recall that a topological surface is of \emph{finite type} if its fundamental group is finitely generated, and of \emph{infinite type} otherwise.

Let us consider a branched covering map $p:S\rightarrow X$ of surfaces. We say that $p$ is \emph{fully ramified} if for every branch point  $b \in X$ of $p$, every point $\tilde{b}$ in the fiber $p^{-1}(b)$ is a critical point, i.e., the ramification number of $p$ at $\tilde{b}$ is at least two. Our main result applies to fully ramified branched covering maps.

\begin{mainthm}[Birman--Hilden property for big mapping class groups]\label{BirmanHildenProperty-Generalization}
   Let $S$ and $X$ be two connected topological surfaces without boundary, where $S$ is either of infinite type or of finite type with negative Euler characteristic. If $p:S\to X$ is a fully ramified branched covering map, then $p$ has the Birman-Hilden property.
\end{mainthm}

As regular branched covering and unbranched covering maps are fully ramified, the above provides the following.

\begin{maincor}\label{coro1}
Let $S$ and $X$ be two connected topological surfaces without boundary, where $S$ is either of infinite type or of finite type with negative Euler characteristic. If $p:S\to X$ is either a regular branched covering or an unbranched covering map, then $p$ has the Birman-Hilden property.   
\end{maincor}

We prove Theorem \ref{BirmanHildenProperty-Generalization} by generalizing the strategy of Winarski's proof of \cite[Theorem 1]{Winarski2015}; see also \cite[Section 9.4.3]{FaMa12}. Our approach uses combinatorial topology of curves on surfaces and relies on the Alexander method for finite and infinite type surfaces. A key step that allows us  to include infinite-sheeted branch covers in our proof is Lemma \ref{Lemma:BranchedCoverOfTheAnnulusStrip}; it gives a description of the structure of branched covers  from the closed annulus and the closed infinite strip to a compact surface with two  boundary components. Furthermore, we work on the category of Klein surfaces to obtain some auxiliary lemmas. This allows us to consider a unified argument which includes infinite-sheeted branched covering maps, and surfaces of finite and infinite type, orientable and non-orientable.

%%%%%%%%%%%%%%%
\subsection{Relation with known results in the literature}  Under the assumption that the surfaces are closed, orientable and hyperbolic,  Winarski proved in \cite[Theorem 1]{Winarski2015} that {fully ramified} branched coverings maps (so necessarily finite-sheeted) satisfy the Birman-Hilden property.  
Her result recovers the case of regular branched covers originally due to Birman and Hilden \cite{BH71,BH72, BH73}, to Zieschang \cite{Zie73}, and to Maclachlan and Harvey  \cite{MacHarvey},  and the case of unbranched covers proved by Aramayona, Leininger and Souto \cite[Proposition 3]{ALS09}. The work of Birman and Hilden also includes regular unbranched covers between closed surfaces that may be non-orientable, including the orientable double cover of a closed non-orientable surface; see \cite[Section 4]{BH72}.  Furthermore, Atalan and Medetogullari \cite[Theorem 1.1]{AM20} extended Winarski's result for  fully ramified branched covers of non-orientable closed hyperbolic surfaces using the orientable double cover.  A survey of the work of Birman and Hilden, and a brief history of the evolution of this theory for finite-sheeted branched covers can be found in \cite{MW21BH}; see also references therein. 

For the case of infinite-sheeted branched covers, a result where the Birman--Hilden property has been explicitly proved is the following. Let $L$ be the \emph{Jacob ladder} and $h:L \rightarrow L$ be the \emph{handle shift} homeomorphism. In  \cite[Lemma 4.1]{SounyaEtAl2021}, Dey, Dhanwani, Patil, and Rajeevsarathy showed that the induced regular unbranched covering map $p_g\colon L\rightarrow S_g$ induced by the action of $\langle h^{g-1}\rangle$ on $L$  satisfies the Birman-Hilden property for every $g\geqslant 2$.

On the other hand, as we explain in  Section \ref{Rmk:original}, the original argument of Birman and Hilden \cite[Theorem 3]{BH72}, \cite[Theorems 1 and 2]{BH73} can be combined with the Alexander method to give a proof of the Birman--Hilden property for unbranched regular covers  $p:S\to X$ over a non-compact hyperbolic base surface $X$. 
This includes the possibility for the surface $X$ to be non-orientable or of infinite type, and for the cover $p$ to have infinitely many sheets. 

All of the above mentioned results are particular cases of Theorem \ref{BirmanHildenProperty-Generalization}. 

\begin{remark}
In \cite[Proposition 1]{Winarski2015}, Winarski gives a necessary condition, the {\it weak curve lifting property}, for when a finite-sheeted branched covering map $p:S\to X$ satisfies the Birman--Hilden property. She applies this condition to obtain examples of non-fully ramified covering maps with finitely many sheets that do not satisfy the Birman--Hilden property. To stablish the weak lifting property, Winarsky uses that the \emph{liftable mapping class group} (see Definition \ref{def:symlift}) is a finite index subgroup of the mapping class group,   which is  true for finite-sheeted branched covers over a surface $X$ of finite type. When the branched cover has infinitely many sheets or the surface $X$ is of infinite type her argument cannot be replicated.   This motivates the following question: \emph{is there a necessary condition, like the curve lifting property,  for when an infinite-sheeted branched cover satisfies the Birman--Hilden property?} 
\end{remark}

%%%%%%%%%%%%%%%%
\subsection{Some applications to geometrically characteristic covers} 

Our Corollary  \ref{coro1} is satisfied, in particular, by unbranched {\it geometrically characteristic covers} (see Definition \ref{def:geomchar}). We use this in Section \ref{Sec:Applications} to obtain injective homomorphisms between certain mapping class groups and surface braid groups, extending known results to the context of non-orientable surfaces and infinite type surfaces. 

Let $B\subset \Sigma$ be closed and discrete. Let us denote by $\Map(\Sigma; B)$ the \emph{mapping class group} of a surface $\Sigma$ \emph{relative to} $B$, which is the group of all isotopy classes (relative to $B$) of homeomorphisms of $\Sigma$ that preserve the set $B$. We require these homeomorphisms to be orientation-preserving if $\Sigma$ is orientable. In the literature, $\Map(\Sigma; B)$ is often called ``big mapping class group'' when the surface $\Sigma$ is of infinite type or $B$ is an infinite set.

Our Proposition  \ref{Res:Injective:Homomorphism}, recovers and extends \cite[Proposition 5]{ALS09}, and it is used to give an alternative proof of \cite[Proposition 3.2]{ALMc2024}, which includes the case when the base surface is non-orientable.

\begin{maincor}[Injections between mapping class groups]\label{Cor:Injection:Characteristic}
Let $S$ and $X$ be connected topological surfaces without boundary, where $S$ is of infinite type or of finite type with negative Euler characteristic. Let $p \colon S \to X$ be an unbranched geometrically characteristic covering map. Fix some $x\in X$  and  some $\widetilde{x} \in p^{-1}(x)$. Then the homomorphism
    \[
        \Map(X;\{x\}) \to \Map(S; p^{-1}(x)), \quad [f] \mapsto [\widetilde{f}]
    \]
    is injective, where $\widetilde{f}$ is the unique lift of $f$ satisfying $\widetilde{f}(\widetilde{x}) = \widetilde{x}$.
\end{maincor} 

An example of a geometrically characteristic cover is given by the \emph{universal abelian cover} $\Sigma^{\text{ab}}$ of a surface $\Sigma$. It is the unbranched regular cover $p:\Sigma^{\text{ab}} \to \Sigma$ corresponding to the commutator subgroup of the fundamental group of $\Sigma$. If $\Sigma$  is a connected surface either of infinite type or of finite type with negative Euler characteristic, then the covering map $p:\Sigma^{\text{ab}} \to \Sigma$  has infinitely many sheets and satisfies the Birman--Hilden property by our Corollary \ref{coro1}. Moreover, the {\it liftable mapping class group} of a geometric characteristic cover  is precisely the whole mapping class group of the base surface of the cover; see Proposition \ref{Prop:Geometrically:LMod=Mod}. Therefore,  it follows from the short exact sequence (\ref{Eq:ExactSequence:MCG's}) that $\Map(\Sigma)$ is isomorphic to a quotient of a subgroup of $\Map(\Sigma^{\text{ab}})$: the \emph{symmetric mapping class group} $\SMap(\Sigma^{\text{ab}})$ with respect to $p$ (see Definition \ref{def:symlift}). 

\begin{maincor}[Mapping class groups as quotients]\label{Cor:MCG_as_Quotients}
    Let $\Sigma$ be a surface of infinite type or of finite type with negative Euler characteristic, and let $\Sigma^{\text{ab}}$ be its universal abelian cover. Then $$\Map(\Sigma)\cong \SMap(\Sigma^{\text{ab}})/H_1(\Sigma;\mathbb{Z}).$$
\end{maincor}

When  $\Sigma$ is orientable, Basmajian and the second author \cite[Section 1.1]{BasHidalgo} obtained the homeomorphism type of its universal abelian cover $\Sigma^{\text{ab}}$; see also \cite[Theorem 1.2]{BCC24} and \cite[Proposition 6.1]{AGH22}. More precisely, the universal abelian cover of the punctured torus is the \emph{flute surface}
; if $\Sigma$ is of finite type with genus at least two and exactly one puncture, then $\Sigma^{\text{ab}}$ is the \emph{spotted Loch Ness monster surface}. The universal abelian cover of the remaining surfaces is the \emph{Loch Ness monster surface} $LNM$.   
In particular, the universal abelian cover of  $LNM$  is again $LNM$. It is interesting to see that, by Corollary \ref{Cor:MCG_as_Quotients}, the mapping class group of $LNM$ is isomorphic to an abelian quotient of one of its subgroups. 
  
On the other hand, the universal abelian cover $S_g^{\text{ab}}$ of a closed orientable surface $S_g$ of genus $g\geq 2$ is also homeomorphic to $LNM$. 
As a consequence of Corollary \ref{Cor:MCG_as_Quotients}, we observe that the corresponding symmetric mapping class group $\SMap(S^{\text{ab}})$  is a finitely generated subgroup of the mapping class group of $LNM$. Furthermore, our Corollary \ref{Cor:Injection:Characteristic} applies and gives a monomorphism from the mapping class group of $S_g^1$, the punctured closed orientable surface of genus $g\geq 2$, into the mapping class group of the {\it spotted} $LNM$.

\medskip

Another interesting geometrically characteristic cover is the \emph{orientable double cover} $p:S\to N$ of a non-orientable surface $N$. When $N$ is of infinite type, we prove in Corollary \ref{Corollary:DoubleCover} that the mapping class group of $N$ can be realized as a subgroup of the mapping class group of $S$. This was proved by Birman and Chillingworth \cite{BC72} for closed surfaces, as an application of Birman--Hilden theory \cite[Section 4]{BH72}. 
For surfaces of finite type, Gon{\c{c}}alves, Guaschi, and Maldonado \cite[Theorem 1.1 and Proposition 2.3]{GGMaldonado2018NonOrientable} give a proof that uses the Birman exact sequences associated to $N$ and $S$ and an application of \cite[Theorem 1.1]{GG12BraidsCovers} by Gon{\c{c}}alves and Guaschi, which shows that the orientable double cover $p:S\to N$ induces a monomorphism of the \emph{braid group $B_k(N)$ on $k$-strands of $N$} into the \emph{braid group $B_{2k}(S)$ on $2k$-strands of $S$}.

Using an argument in the reverse direction, we apply our  Corollary \ref{Corollary:DoubleCover} to prove in Corollary \ref{Cor:Injectivity:Braids}  the injectivity of the corresponding surface braid groups, including the infinite type case; see also Corollary \ref{Cor:Braids2}. These results are summarized in the following statement:

\begin{maincor}[Mapping class groups and braid groups of a non-orientable surface and its orientable double cover]\label{Cor:OrientableDoubleCover}
Let $N$ be a non-orientable connected surface, with empty boundary, of infinite type or of finite type with negative Euler characteristic, and $p:S\to N$ its {\it orientable double cover}. Then the natural homomorphisms induced by $p$ 
$$\Map(N)\rightarrow \Map(S)\text{\ \ \ \ \ \  and \ \ \ \ \ \  } B_k(N) \rightarrow B_{2k}(S)$$
are injective for $k\geqslant 1$. 
\end{maincor}

In \cite[Theorem 1.1]{KKuno}, Katamaya and Kuno proved that for closed surfaces the injective homomorphism $\Map(N)\rightarrow \Map(S)$ is a quasi-isometric embedding. It would be interesting to understand what can be said for surfaces of infinite type.

%%%%%%%%%%%%%%%
\subsection*{Outline of the paper} In Section \ref{Sec:Prel}, we introduce the necessary notation and preliminaries on mapping class groups, the Alexander method, branched covers and the equivalent statements of the Birman--Hilden property. In Section \ref{Sec:Klein}{\color{blue},} we work on the category of Klein surfaces and prove some auxiliary results that are needed to obtain our main result  Theorem \ref{BirmanHildenProperty-Generalization}, which we prove in Section \ref{Sec:MainProof}. We end the paper by discussing some applications in Section \ref{Sec:Applications}.

%%%%%%%%%%%%%%%%%%%
%%%%%%%%%%%%%%%%%%%
\section{Preliminaries}\label{Sec:Prel}

%%%%%%%%%%%%%%%%%%%
\subsection{Mapping class group of surfaces} 
Let $\Sigma$ be a connected topological surface and let $B$ be a closed and discrete subset in the interior of $\Sigma$. Here we consider the set $B$ as \emph{marked points} in the surface $\Sigma$. The boundary (possibly empty) of $\Sigma$  is denoted by $\partial \Sigma$, and we always assume that $\partial \Sigma$  is compact. 

If $\Sigma$ is non-orientable, we denote by  $\Homeo(\Sigma;B)$ the group of all homeomorphisms $f:\Sigma \to \Sigma$ that preserve $B$ as a set and restrict to the identity on $\partial \Sigma$. If $\Sigma$ is orientable, $\Homeo(\Sigma;B)$ denotes the group of all orientation-preserving homeomorphisms of $f:\Sigma \to \Sigma$ that preserve the set $B$ and restrict to the identity on $\partial \Sigma$.    

The \emph{mapping class group} of $\Sigma$ \emph{relative to} $B$, denoted by $\Map(\Sigma;B)$, is the group of all isotopy classes of elements in $\Homeo(\Sigma;B)$; we require that all homeomorphisms and
isotopies preserve $B$ and fix $\partial \Sigma$ pointwise. We have the natural projection map 
    \begin{equation}\label{Eq:NaturalProjection}
        \rho=\rho_{\Sigma; B}:\Homeo(\Sigma;B)\rightarrow \Map(\Sigma;B),  \ \ \ \ \ \ \ f\mapsto [f].    
    \end{equation} 
If the set of marked points $B$ is empty, we simply write $\Homeo(\Sigma)$ and $\Map(\Sigma)$.

%%%%%%%%%%%%%%%%%%
\subsection{The Alexander method} 
A \emph{simple closed curve} in a surface $\Sigma$ is an embedding of the unitary circle into $\Sigma$. A simple closed curve in  $\Sigma$  is \emph{essential} if it does not bound a disc, a once-punctured disc, an annulus or a Möbius strip. We say that $\Sigma$ is \emph{sporadic} if it has no essential simple closed curves. One can verify that a surface $\Sigma$ with empty boundary is sporadic if and only if $\Sigma$ is homeomorphic to either a sphere with at most three punctures or a projective plane with at most one puncture.

\begin{definition}[Alexander system]
    Let $\Gamma = \{\gamma_i \}_{i\in \mathbb{N}}$ be a collection of essential simple closed curves on a surface $\Sigma$ with empty boundary. We say that $\Gamma$ is an \emph{Alexander system} if it satisfies the following conditions:

    \begin{enumerate}
        \item \emph{(minimal position)} the elements in $\Gamma$ are in pairwise minimal position\footnote{Two simple close curves are in \emph{minimal position} it they do not form a \emph{bigon}, see \cite[Section 1.2.4]{FaMa12} for details.},
        \item \emph{(distinct isotopy classes)} for $\gamma_i,\gamma_j\in \Gamma$ with $i\neq j$, we have that $\gamma_i$ is not isotopic to $\gamma_j$,
        \item \emph{ (no triple intersections)} for all distinct $i, j, k\in \N$, at least one of the following sets is empty: $\gamma_i \cap \gamma_j$, $\gamma_j \cap \gamma_k$, $\gamma_k \cap \gamma_i$ and, 
        \item \emph{(local finiteness)} for any finite type subsurface $S$ of $\Sigma$, up to isotopy, $S$ interects a finite number of curves in $\Gamma$.
    \end{enumerate}
\end{definition}

Let $\mathcal{A}$ be a collection of curves on a surface $\Sigma$ with empty boundary. We say that $\mathcal{A}$ \emph{fills} $\Sigma$ if  $\Sigma\setminus \mathcal{A}$ is the disjoint union of open discs and once-punctured open discs. 

\begin{definition}[k-stable Alexander system]
Let $\Gamma$ be an Alexander system on a surface $\Sigma$ with empty boundary and $k\geqslant 1$. We say $\Gamma$ is a $k$-\emph{stable Alexander system} if $\Gamma$ fills $\Sigma$ and every $f\in \Homeo(\Sigma)$  that preserves the isotopy classes of elements in $\Gamma$ satisfies that $f^k$ is isotopic to the identity.
\end{definition}

If $\Sigma$ is an orientable finite type surface, in \cite[Proposition 2.8]{FaMa12} it is shown that any Alexander system that fills $\Sigma$ is $k$-stable for some $k$. The proof also works in the setting of non-orientable finite type surfaces.  Moreover, for any finite type surface it is possible to exhibit $k$-stable Alexander systems for $k$ at most two. Examples of $1$-stable Alexander systems are exhibited in \cite[Theorem 1.1]{alexander-infinite} for orientable infinite type surfaces and in \cite[Theorem A]{Hidber-Hernández2021} for non-orientable infinite type surfaces. The existence of $k$-stable Alexander systems is the so-called \emph{Alexander's method} and will be an important ingredient in the proof of our main result Theorem \ref{BirmanHildenProperty-Generalization}.

\begin{thm}[Alexander's method for surfaces]\label{thm:AlexanderMethodInfiniteSurfaces}
    Let $\Sigma$ be a non-sporadic topological surface with empty boundary. 
    \begin{enumerate}
        \item If $\Sigma$ is of finite type, then it has a $2$-stable Alexander system.
        \item If $\Sigma$ is of infinite type, then it has a 1-stable Alexander system. 
    \end{enumerate}
\end{thm}

See also \cite{Shapiro2021} for a detailed discussion about the Alexander method.

%%%%%%%%%%%%%%%%%%%%
\subsection{Birman-Hilden theory for branched covers}\label{Sec:BH}
Let us recall the terminology needed to define the Birman--Hilden property and some equivalent statements.

\begin{definition}[Branched covering map]
Let $S$ and $X$ be topological surfaces and let 
$p:S \rightarrow X$ be a surjective continuous map.

\begin{enumerate}
\item We say that the map $p$ is a \emph{branched covering map} or \emph{branched cover} if there exists a closed and discrete subset $B \subset \text{int } X$, which might be empty, called the \emph{branch locus} of $p$, with the following properties: 

\begin{enumerate}
    \item[a)] $p\vert_{S^\ast}: S^\ast \to X^\ast$ is a covering map where $S^\ast:=S\setminus p^{-1}(B)$ and $X^\ast:=X\setminus B$, and 
    \item[b)] each point $b\in B$ has a neighborhood $U_b$, homeomorphic to an open disk, such that $p^{-1}(U_b)$ is a disjoint union of a collection $\{V_j\}_{j\in J}$ of open sets, each one homeomorphic to an open disk cointaining exactly one point $\tilde{b}_j\in p^{-1}(b)$, and $p\vert_{V_j}:V_j\to U_b$ is a finite branched cover for each $j\in J$, that is, it is topologically equivalent to a branched cover $z\in \mathbb{D}^2\mapsto z^{m_j}\in \mathbb{D}^2$, where $m_j \geq 1$ (with at least one of them bigger or equal to two) is an integer called the \emph{ramification number of $p$ at $\tilde{b}_j$.}
\end{enumerate}

\item If the set $B$ is empty, we say that $p$ is  \emph{unbranched}.

\item If $p$ is a branched covering map whose restriction  $p\vert_{S^\ast}: S^\ast \to X^\ast$ is a regular covering map, then we say that $p$ is 
a \emph{regular branched covering map}. 

\item The \emph{Deck group} of $p$, denoted by $\mathrm{Deck}(p)$, is the group of deck transformations of the covering map $p\vert_{S^\ast}$.
\end{enumerate}
\end{definition}

\begin{remark}
    The branch locus $B$ of a branched covering map $p:S\rightarrow X$ is closed and discrete in $X$. Therefore, $p^{-1}(B)$ is also discrete and closed in $S$. Each point of $B$ is called a {\it branch point} of $p$.
\end{remark}

\begin{definition}[Fully ramified cover]\label{Def:fully}
Let $p:S\rightarrow X$ be a branched covering map. An element in $p^{-1}(B)$ is a \emph{critical point} if its ramification number is at least two; otherwise, it is called an \emph{unramified point}. We say that $p$ is \emph{fully ramified} if each element in the preimage of the branch locus of $p$ is a critical point. 
\end{definition}

Let $p: S \to X$ be a branched covering map of topological surfaces.  We say that a function $f : S \to S$ is \emph{fiber-preserving} if for each $x \in X$ there is $y \in X$ so that $f(p^{-1}(x)) = p^{-1}(y)$. In this situation, there is a function $\bar{f}:X \to X$ such that $\bar{f} \circ p= p \circ f$. This means also that $\bar{f}(B) \subset B$.
If $f$ is a homeomorphism, then $\bar{f}$ is also a homeomorphism preserving $B$.

Let $\mathrm{SHomeo}(S)$ denote the group of all fiber-preserving homeomorphisms of $S$. We call this group the \emph{symmetric homeomorphism group} of $S$ with respect to $p:S \to X$ (in what
follows, whenever $p$ is clear, we will just say {\it symmetric homeomorphism group}). On the other hand, denote by $\mathrm{LHomeo}(X;B)$ the group of all \emph{liftable} homeomorphisms of $X$ that preserve the branch locus $B$. As any fiber-preserving homeomorphism of $S$ projects under $p$ to a homeomorphism of $X$ that preserves the branch locus $B$, and any liftable homeomorphism of $X$ (so preserving $B$) defines a fiber-preserving homeomorphism of $S$, there is a well-defined surjective group homomorphism 
    \begin{equation}\label{eqn:Phi:Homeo's}
        \widetilde{\Phi}: \SHomeo(S)\longrightarrow \LHomeo(X;B)
    \end{equation}

and a short exact sequence of groups 
    \begin{equation} \label{Eq:ExactSequence:Homeo's}
        1\rightarrow \Deck(p) \rightarrow \SHomeo(S) \xrightarrow{\widetilde{\Phi}} \LHomeo(X;B) \rightarrow 1.
    \end{equation}

\begin{definition}[Symmetric and liftable mapping class group]\label{def:symlift} Let $p: S \to X$ be a branched covering map of topological surfaces.  We define the \emph{symmetric mapping class group} of $S$ by
\[ 
    \SMap(S):=\rho_S(\SHomeo(S)),
\]
and the \emph{liftable mapping class group} of $X$ by
\[
    \LMap(X;B):=\rho_{X;B}(\LHomeo(X;B)).
\]
Here $\rho_S$ and $\rho_{X,B}$ denote the corresponding projection from the homeomorphism group of the surface to its mapping class groups as defined in (\ref{Eq:NaturalProjection}). Through this paper, we denote by $D$ to the image of $\Deck(p)$ under $\rho_S$, i.e., $D:=\rho_S(\Deck(p))$. 
\end{definition}

Notice that two homeomorphisms of $S$ are identified in $\SMap(S)$ if they diﬀer by an isotopy that is not necessarily ﬁber-preserving. Therefore, the group of homomorphims defined in (\ref{eqn:Phi:Homeo's}) $\widetilde{\Phi}\colon \SHomeo(S) \to \LHomeo(X;B)$  a priori does not induce a well-defined homomorphism $\Phi\colon \SMap(S)\to \LMap(X;B)$. \emph{Birman-Hilden theory} studies which conditions on the branch covering map $p$ ensure the existence of a well-defined homomorphism $\Phi$.

\begin{definition}[Birman--Hilden property]
Let $p:S\to X$ be a branched covering map between topological surfaces. Then $p$ has the \emph{Birman--Hilden property} if every fiber-preserving homeomorphism of $S$ that is isotopic to $Id_S$  is also isotopic to $Id_S$ through fiber-preserving homeomorphisms, where $Id_S$ denotes the identity map on $S$.  
\end{definition}

\begin{remark}\label{Rmk:NonExample:BH} Some examples of branched covers that do not satisfy the Birman–Hilden property are given below. See also \cite[Section 3.2]{Winarski2015} and references therein. 

\begin{enumerate}
    \item Let $S$ be the torus $T^2$, and consider the branched covering map $p : T^2 \to X$ corresponding to the hyperelliptic involution of the torus; see  \cite[Section 2]{MW21BH}.

    \item Another example can be constructed from the orientable double cover of the Klein bottle $\mathbb{K}$. To describe it explicitly, consider the torus as $T^2 = S^1 \times S^1 \subset \mathbb{C} \times \mathbb{C}$. Define the involution $\sigma : T^2 \to T^2$ by  
    \(
    \sigma(z, w) = (e^{i\pi}z, \overline{w}),
    \)  
    for all $(z, w) \in T^2$, where $\overline{w}$ denotes the complex conjugation. Then $\mathbb{K} = T^2 / \langle \sigma \rangle$ and the projection map $p : T^2 \to \mathbb{K}$ is a covering map. Now define $f \in \Homeo(T^2)$ by
    \[
    f(z, w) := 
    \begin{cases}
    (z, z^2 w), & \text{if } \Im(z) \leqslant 0, \\
    (z, {\overline{z}}^{2} w), & \text{if } \Im(z) \geqslant 0.
    \end{cases}
    \]
    Here $ \Im(z)$ denotes the imaginary part of $z$. Note that $f$ commutes with $\sigma$ and it is an orientation-preserving homeomorphism of $T^2$. Moreover, $f$ induces the identity on $H_1(T^2; \mathbb{Z})$, the first homology group of the torus. Therefore, $f$ is isotopic to $id_{T^2}$. However, the homeomorphism $\bar{f} : \mathbb{K} \to \mathbb{K}$ induced by $f$ via the covering map $p$ does not induce the identity on $H_1(\mathbb{K}; \mathbb{Z})$, and hence $\bar{f}$ is not isotopic to $id_{\mathbb{K}}$. In conclusion, the covering map $p: T^2 \to \mathbb{K}$ does not satisfy the Birman--Hilden property.

    \item Let $S_3$ be the orientable closed surface of genus $3$. Consider the threefold simple branched cover $p:S_3 \to S^2$ over the $2$-sphere $S^2$ described in \cite[Section 3.2 and Fig. 3]{Winarski2015}, and let $q:S_3^{\text{ab}} \to S_3$ be the universal abelian cover of $S_3$, that is, it is the unbranched regular cover corresponding to the commutator subgroup of the fundamental group of $S_3$. Notice that $q$ is a characteristic cover of infinite degree, and by \cite[Corollary 4]{BasHidalgo} (see also \cite[Theorem 1.2]{BCC24}) the surface $S_3^{\text{ab}}$ is homeomorphic to the Loch Ness monster surface. Since the branched cover $p:S_3 \to S^2$ does not have the Birman–Hilden property, it follows that the composition $p \circ q:S_3^{\text{ab}} \to S^2$ is an infinitely-sheeted branched cover that does not satisfy the Birman–Hilden property.
\end{enumerate}

\end{remark}
We repeatedly use the following equivalent statements of the Birman–Hilden property. Given Proposition \ref{Prop:DectTransfIsotopicId} (from next section), this result is essentially \cite[Proposition 3.1]{MW21BH}. It  applies in the level of generality we consider, which includes infinite type surfaces and branched covering maps that may have infinitely many sheets.  In Section \ref{Sec:MainProof}, we prove Theorem \ref{BirmanHildenProperty-Generalization}  by verifying item \eqref{item:BH3} among the equivalent conditions listed below.

\begin{proposition}[Equivalent formulations of the Birman-Hilden property]\label{Prop:Equivalence:BirmanHilden}
Let $S$ be a topological surface without boundary that is either of infinite type or of finite type with negative Euler characteristic. If $p : S \to X$ is a branched covering map, then the following statements are equivalent:
\begin{enumerate}
    \item $p$ has the Birman--Hilden property.
    
    \item $\mathrm{SHomeo}(S)\cap \Homeo_0(S)$ is path-connected in the group $\Homeo_0(S)$ of all homeomorphisms of $S$ that are isotopic to $Id_S$ endowed with the compact-open topology.
    
    \item \phantomsection\label{item:BH3} The map $\Theta:\LMap(X;B) \to \SMap(S)/D$  is injective.
    
    \item \phantomsection\label{item:BH4} The map $\Phi:\SMap(S) \to \LMap(X;B)$ is well defined. 
    
    \item $\SMap(S)/D \cong \LMap(X;B)$.
\end{enumerate}
\end{proposition}

\begin{remark}
    In Birman-Hilden theory, the branch locus $B$ and its preimage $p^{-1}(B)$ are treated differently. The branch points of the branched covering map $p:S\to X$ are considered as marked points on the surface $X$; this is because any liftable homeomorphism should leave the set of branch points invariant. However, points in $p^{-1}(B)$ are considered as ordinary points in the surface $S$, not as marked points. We make this distinction explicit in the notation. 
\end{remark}

As we will see from Proposition \ref{Prop:DectTransfIsotopicId}, we have that $D\cong \Deck(p)$ when $S$ is a topological surface without boundary that is either of infinite type or it has negative Euler characteristic. With this identification, it follows from Proposition \ref{Prop:Equivalence:BirmanHilden} that the branched cover $p : S \to X$ satisfies the Birman-Hilden property if and only if the short exact sequence (\ref{Eq:ExactSequence:Homeo's}) induces a short exact sequence
\begin{equation} \label{Eq:ExactSequence:MCG's}
        1\rightarrow \Deck(p) \rightarrow \SMap(S) \xrightarrow{\Phi} \LMap(X;B) \rightarrow 1.
\end{equation}

%%%%%%%%%%%%%%%%%%%%%
%%%%%%%%%%%%%%%%%%%%%
\section{Klein surfaces and auxiliary results}\label{Sec:Klein}
In this section, we prove some auxiliary results that are needed to obtain Theorem \ref{BirmanHildenProperty-Generalization} in Section \ref{Sec:MainProof}. In order to state our results in a way that applies in general for non-orientable surfaces, we need to move to the category of Klein surfaces which includes the category of Riemann surfaces.

\subsection{Klein surfaces and some double covers}\label{SubSec:KleinSurfaces} 
A \emph{Klein surface} is a topological surface $\Sigma$, possibly with non-empty boundary, with a \emph{dianalytic} structure, that is, a structure in which the transition functions are either analytic or anti-analytic. 
Since we are considering surfaces possibly with non-empty boundary, the notion of dianalyticity of a function is taken in its extended form admitting open subsets of the positive hyperbolic semi-plane $\mathbb{H}^+:=\{a+bi\in \mathbb{C} \mid b\geq 0 \}$ as domain.  A Klein surface might be orientable or not. If it is orientable, and the dianalytic structure is such that all changes of coordinates are analytic, then it is called a \emph{Riemann surface}.   We refer the reader to see \cite[Chapter 1]{alling2006} for a preliminary study of Klein surfaces and morphisms between Klein surfaces.  Given a Klein surface $\Sigma$, we denote by $\mathrm{Aut}(\Sigma)$ the \emph{automorphism group} of $\Sigma$. 

Any topological surface, possibly with non-empty boundary, admits a Klein surface structure. It is a consequence of the existence of isothermal coordinates in dimension two, see for instance \cite[Section 1.5.1 and Theorem 4.24]{IT92Teich}. 

By the Poincaré-Koebe Uniformization Theorem, every Klein surface $\Sigma$ can be represented as $\mathcal{U}/\Gamma$ where $\mathcal{U}$ is a simply connected Riemann surface and $\Gamma$ is a discrete subgroup of isometries of $\mathcal{U}$, without elliptic elements, that might contain reflections whose set of fixed points are projected to the boundary components of $\Sigma$. Notice that if $\Sigma$ is non-orientable or it has non-empty boundary, then $\Gamma$ necessarily has anti-analytic elements. By a slight abuse of notation, we will often refer to the Klein surface simply as $\Sigma$. \\

The following constructions will be used in our arguments below.\\

\noindent {\bf The complex double cover}. Suppose $\Sigma={\mathcal U}/\Gamma$ is non-orientable or $\partial \Sigma$ is not empty. Then $\Gamma$ must have anti-analytic automorphisms. Let $F=\Gamma^{+}$ be the index two subgroup of $\Gamma$ consisting of the analytic elements of $\Gamma$. The quotient space $R:=\mathcal{U}/F$ is a Riemann surface without boundary, and the inclusion $F \lhd \Gamma$ induces a degree two, possibly branched, covering map $f:R \to \Sigma$, with Deck group generated by an anti-analytic involution $\tau:R\to R$. This involution has fixed points if and only if $\Sigma$ has a non-empty boundary. In particular, the cover $f$ is unbranched outside the image under $f$ of the set of fixed points of $\tau$. The surface $R$ is called the \emph{complex double cover} of $\Sigma$; for further details we refer the reader to the proof of \cite[Theorem 1.6.1]{alling2006}. We will use the complex double cover of a Klein surface in the proof of Theorem \ref{Prop:AutIsDiscrete}.

If $\Sigma$ is non-orientable and $\partial \Sigma = \emptyset$, then the complex double cover of $\Sigma$ coincides with its so-called \emph{orientable double cover}. It is also possible to define the orientable double cover of a surface with non-empty boundary, which differs from the complex double cover in this situation; we refer the reader to \cite[Theorem 1.6.7]{alling2006} for a detailed explanation. We consider the orientable double cover of a non-orientable surface with empty boundary in Section \ref{sec:orientable double cover}.\\

\noindent {\bf The Schottky double cover}. 
 Let $\Sigma={\mathcal U}/\Gamma$ be a Klein surface with non-empty boundary. 
The {\it Schottky double cover} $\widehat{\Sigma}$ of $\Sigma$, as defined in \cite{alling2006,Costa2012}, is a Klein surface with empty boundary and of the same orientability type as $\Sigma$. 

As $\Sigma$ has non-empty boundary, $\Gamma$ contains reflections, and it contains index two subgroups containing none of its reflections. One of these index two subgroups, say $K$, satifies that $\widehat{\Sigma}={\mathcal U}/K$.
So, there is a degree two branched covering map $f:\widehat{\Sigma} \to \Sigma$ with Deck group $G=\Gamma/K$ generated by an involution $\tau$. The image under $f$ of the set of fixed points of $\tau$ coincides with the boundary of $\Sigma$, and $\widehat{\Sigma}$ is described as the disjoint union of two copies of $\Sigma$ identifying their boundaries. The connected components of the set of fixed points of $\tau$ are called its \emph{ovals}.

Our argument to prove Lemma \ref{Lemma:BranchedCoverOfTheAnnulusStrip} only uses the Schottky double cover of the closed infinite strip\footnote{The \emph{closed infinite strip} is the surface homeomorphic to $\mathbb{R} \times [0,1]$.}, the closed annulus, and finite type hyperbolic Klein surfaces. For these cases, we describe next the subgroup $K$, the surface $\widehat{\Sigma}$ and the Deck group $G$.

\begin{enumerate}
    \item Let $\Sigma$ be a closed infinite strip. In this case, ${\mathcal U}={\mathbb C}$ and $\Gamma=\langle a,b \rangle \cong {\mathbb Z}_{2} * {\mathbb Z}_{2}$, where $a(z)=-\overline{z}$ and $b(z)=-\overline{z}+2$. We set $K=\langle ab \rangle \cong {\mathbb Z}$; which is of index two in $\Gamma$. The Schottky double cover $\widehat{\Sigma}={\mathcal U}/K$ is the punctured plane (the sphere minus two points), and the group $G$ is generated by the reflection $z \mapsto 1/\overline{z}$.

    \item Now suppose $\Sigma$ is the closed annulus. In this case, ${\mathcal U}={\mathbb C}$ and $\Gamma=\langle a,b,e \rangle$, where $a(z)=-\overline{z}$, $b(z)=-\overline{z}+r$, where $r>0$, and $e(z)=z+i$. We set $K=\langle ab,e \rangle \cong {\mathbb Z}^2$; which is of index two in $\Gamma$. The Schottky double cover $\widehat{\Sigma}={\mathcal U}/K$ is the torus, and $G$ is generated by an involution with two ovals.

    \item  Let $\Sigma$ be a hyperbolic Klein surface of finite type with non-empty boundary. There are two cases to consider, depending on the orientability type of $\Sigma$.

    (i) If $\Sigma$ is non-orientable of genus $g \geq 1$ and exactly $k\geq 1$ boundary components with $g-2+k>0$, then $\Gamma$ is generated by $g$ glide-reflections $d_{1},\ldots,d_{g}$, $k$ reflections $c_{1},\ldots,c_{k}$, and $k$ hyperbolic elements $e_{1},\ldots,e_{k}$, satisfying the following relations:
    $c_{i}^{2}=1$, $e_{i}c_{i}e_{i}^{-1}c_{i}=1$, and $e_{1}\cdots e_{k} d_{1}^{2}d_{2}^{2} \cdots d_{g}^{2}=1$. Consider the surjective homomorphism $\theta:\Gamma \to {\mathbb Z}_{2}=\langle \tau \rangle$, defined by $\theta(c_{i})=\tau$, $\theta(e_{i})=\theta(d_{j})=1$ for all $i,j$. Then $K:=\ker(\theta)$ has index two in $\Gamma$.

    (ii) If $\Sigma$ is orientable of genus $g\geq 0$ and $k\geq 1$ boundary components such that $2g-2+k >0 $, then $\Gamma$ is generated by $2g$ hyperbolic elements, $a_{1},b_{1},\ldots,a_{g},b_{g}$, 
    $k$ reflections $c_{1},\ldots,c_{k}$, and $k$ hyperbolic elements $e_{1},\ldots,e_{k}$, satisfying the following relations:
    $c_{i}^{2}=1$, $e_{i}c_{i}e_{i}^{-1}c_{i}=1$, and $e_{1}\cdots e_{k} [a_{1},b_{1}] \cdots [a_{g},b_{g}]=1$. 
    In this case, we define similarly $\theta$ by sending $a_{j}$, $b_{j}$, $e_{i}$ to $1$, and $c_{i}$ to $\tau$ for all $i,j$, and set $K$ as its kernel.

 In these cases, the Schottky double cover is $\widehat{\Sigma}:={\mathcal U}/K$, and $G:=\Gamma/K \cong {\mathbb Z}_{2}$ is generated by an involution with exactly $k$ ovals. 
\end{enumerate}

The next result shows that any automorphism of a Klein surface isotopic to the identity must be equal to the identity.

\begin{thm}\label{Prop:AutIsDiscrete}
 Let  $\Sigma$ be a Klein surface, possibly with non-empty boundary, either of infinite type or of finite type with negative Euler characteristic. If $\varphi\in \mathrm{Aut}(\Sigma)$ is isotopic to the identity $Id_{\Sigma}$ then $\varphi=Id_{\Sigma}$.
\end{thm}

\begin{proof}
Observe that it is enough to show the result when $\Sigma$ is a Riemann surface with empty boundary. Indeed, suppose $\Sigma$ is a Klein surface but not a Riemann surface and suppose the result is true for Riemann surfaces with empty boundary. Let $f:R\to \Sigma$ be the complex double cover of $\Sigma$. We recall that by construction, $R$ is a Riemann surface with empty boundary. Now, we can choose an analytic lift $\widetilde{\varphi}:R\to R$ of $\varphi$ isotopic to $Id_{R}$. Given that $R$ is a Riemann surface, we get that $\widetilde{\varphi}=Id_{R}$ and therefore $\varphi=Id_{\Sigma}$.

Now suppose $\Sigma$ is a Riemann surface with universal cover $\pi:\mathbb{H}^2\to \Sigma$ and Deck group $\Gamma$ which is a discrete and torsion-free subgroup (i.e., a \emph{Fuchsian group}) of $\Aut(\mathbb{H}^2)$ such that $\Sigma$ is isometric to the quotient space $\mathbb{H}^2/\Gamma$. Take $\widetilde{\varphi}:\mathbb{H}^2 \to \mathbb{H}^2$ a lift  of $\varphi$. Since $\varphi$ is isotopic to $Id_{\Sigma}$, up to precomposition of $\widetilde{\varphi}$ with an element of $\Gamma$, we can assume that $\widetilde{\varphi}$ is isotopic to $Id_{\mathbb{H}^2}$. Notice that $\widetilde{\varphi}$ is an automorphism of $\mathbb{H}^2$ that is contained in the normalizer of $\Gamma$ in $\mathrm{Aut}(\mathbb{H}^2)$. Therefore, $\widetilde{\varphi}$ preserves the type (hyperbolic or parabolic) of elements of the Deck group. Following Tappu's argument in the proof of \cite[Proposition 4.1]{Tappu2023}, $\widetilde{\varphi}$ has a continuous extension $\widetilde{\varphi} \cup \partial \widetilde{\varphi}:\mathbb{H}^2\cup \overline{ \Lambda(\Gamma)}\to \mathbb{H}^2\cup \overline{ \Lambda(\Gamma)}$ where $\Lambda(\Gamma)\subseteq \mathbb{S}^1$ denotes the \emph{limit set} of $\Gamma$. Moreover, because  $\widetilde{\varphi}$ is a Möbius transformation, the extension $\partial \widetilde{\varphi}$ coincides with the natural action of $\widetilde{\varphi}$ on $\mathbb{S}^1$ restricted to $\Lambda(\Gamma)$. Now, the fact that $\widetilde{\varphi}$ is isotopic to the identity implies that $\partial \widetilde{\varphi}$ is the identity on $\Lambda(\Gamma)$, see \cite[Proposition 4.1 item (4)]{Tappu2023}.  Given the surface $\Sigma$ is of infinite type or of finite type with negative Euler characteristic, the limit set of $\Gamma$ has more than three points, in fact, it is infinite. Therefore, we obtain that $\widetilde{\varphi}$ is necessarily the identity on $\mathbb{H}^2$ and therefore $\varphi=Id_{\Sigma}$.
\end{proof}

%%%%%%%%%%%%%%%%%%%%%%%%
\subsection{Klein surfaces and branched covers}
Let $p \colon S \to X$ be a branched covering map between surfaces. For any dianalytic structure on $X$ we can pullback the Klein structure on $S$ such that the covering $p\colon S\to X$ is a dianalytic map, see  \cite[Corollary 1.4.5 $\&$ Theorem 1.5.2]{alling2006}. Particularly, this gives us a local isometry between the two Klein surfaces. 

\begin{proposition}\label{Prop:DectTransfIsotopicId}
Let $S$ be a topological surface, possibly with non-empty boundary, that is either of infinite type or of finite type with negative Euler characteristic. Consider a branched covering map  $p \colon S \to X$, possibly with infinitely many sheets.  If a deck transformation $\varphi \colon S \to S$ is isotopic to the identity $\mathrm{Id}_S$, then $\varphi = \mathrm{Id}_S$. In particular, $D \cong \Deck(p)$.

\end{proposition}
\begin{proof}
 Endow both $X$ and $S$ with dianalytic structures such that the covering map $p \colon S \to X$ becomes a dianalytic map. In this setting, the group $\Deck(p)$ acts by automorphisms of the Klein surface $S$. Therefore, by Theorem \ref{Prop:AutIsDiscrete}, any deck transformation $\varphi$ isotopic $\mathrm{Id}_S$ must be the identity.
\end{proof}

Given a topological surface $\Sigma$ consider the projection  $\rho:\Homeo(\Sigma)\to \Map(\Sigma)$ from (\ref{Eq:NaturalProjection}).

\begin{definition}
Let $G$ be a subgroup of $\Map(\Sigma)$. We say that $G$ is \emph{realizable} if $\Sigma$ admits a dianalytic structure such that there exists a subgroup $\widetilde{G}$ of $\Aut(\Sigma)\subset \Homeo(\Sigma)$ isomorphic to $G$ that satisfies $\rho(\widetilde{G})=G$.
\end{definition}

\begin{remark}
    If $\Sigma$ is a Klein surface that has negative Euler characteristic or is of infinite type, the group $\Aut(\Sigma)$ acts properly discontinuously on $\Sigma$ (see \cite[Theorem 3.1.4]{DelecroixValdez2024}), and therefore it is countable. Hence, any realizable subgroup of $\Map(\Sigma)$ is countable. 
\end{remark}

\begin{proposition}\label{Prop:KernelThetaTorsionFree}
Let $S$ be a topological surface that is either of infinite type or of finite type with negative Euler characteristic. Consider a branched covering map  $p \colon S \to X$  possibly with infinitely many sheets,  and suppose that every finite cyclic subgroup of $\LMap(X;B)$ is realizable. Then the kernel of $\Theta:\LMap(X;B) \to \SMap(S)/D$ is torsion-free.
\end{proposition}

\begin{proof}
Consider an element $\varphi\in\mathrm{Ker}(\Theta)$ of finite order $n$. Since every finite cyclic subgroup of $\LMap(X;B)$ is realizable, there is a dianalytic structure on $X$ and $f \in  \Aut(X)$ of order $n$ with $\varphi=[f]$.  Let $\widetilde{f}$ be a lift of $f$. By pulling back to $S$ the dianalytic structure of $X$, the covering map $p:S\to X$ is a local isometry and $\widetilde{f} \in \Aut(S)$. Up to pre-composition by a deck transformation, we can assume that $\widetilde{f}$ is isotopic to $Id_S$. Using the Proposition \ref{Prop:DectTransfIsotopicId}, we conclude that $\widetilde{f}$ coincides with $Id_S$. It follows that $f=Id_X$ and $\varphi$ must be trivial.\end{proof}

It is well known that any finite subgroup of the mapping class group of a surface of finite type is realizable (c.f. Nielsen realization). This was proven for closed orientable surfaces in \cite[Theorems 4 and 5]{Ker83}, for punctured surfaces in \cite[Theorem 6.1]{Wol87}, and for non-orientable surfaces in \cite[Theorem 1.4]{ColinXico2024}. In particular, finite cyclic subgroups are realizable. Therefore, the following result is a direct consequence of Proposition \ref{Prop:KernelThetaTorsionFree} and extends \cite[Proposition 4]{Winarski2015}. 

\begin{corollary}\label{Corollary:ThetaFreeTorsionFiniteTypeSurfaces}
    Let $p:S\to X$ be a branched cover where $S$ is of infinite type or has negative Euler characteristic and $X$ is a surface of finite type. Then the kernel of $\Theta:\LMap(X;B) \to \SMap(S)/D$ is torsion-free.
\end{corollary}

\begin{remark}
In \cite[Theorem 2]{ACCL21Nielsen},
Afton--Calegari--Chen--Lyman proved that infinite type orientable surfaces satisfy Nielsen realization, therefore Corollary \ref{Corollary:ThetaFreeTorsionFiniteTypeSurfaces} also holds when $X$ is an orientable surface of infinite type. To the best of our knowledge, an analogous result for infinite type non-orientable surfaces has not yet been established. However, the proof of our main result only needs Corollary \ref{Corollary:ThetaFreeTorsionFiniteTypeSurfaces} as stated, since it relies on the existence of $1$-stable Alexander systems in the infinite type setting.
\end{remark}

In Winarski's proof of \cite[Theorem 1]{Winarski2015}, a key step involves analyzing pairs of isotopic simple closed curves on the surface $S$, which bound an embedded annulus. When considering infinite-sheeted branch covers, we must also consider the case where such curves bound an infinite strip, rather than a compact annulus. 
To address this, we finish this section with a description of the structure of branched covers (possibly with infinitely many sheets) from the closed annulus or the closed infinite strip to a compact surface with two  boundary components. This result will be an essential step in the proof of Theorem \ref{BirmanHildenProperty-Generalization}.

\begin{lemma}\label{Lemma:BranchedCoverOfTheAnnulusStrip}
    Let $p:S\to X$ be a branched covering map. Suppose that $S$ is homeomorphic to the closed annulus or the closed infinite strip, and $X$ is a compact surface with two boundary components. Then $X$ is homeomorphic to the closed annulus and the map $p$ is unbranched.
\end{lemma}

\begin{proof}
    We may endow both $X$ and $S$ with dianalytic structures such that:

\begin{enumerate}
    \item the covering map $p \colon S \to X$ becomes a dianalytic map;
    
    \item if $S$ is homeomorphic to the closed annulus, then (with the given dianalytic structure) it is isomorphic to $\{z \in {\mathbb C}: 1 \leq |z| \leq r\}$, for a suitable $r>1$;
    
    \item if $S$ is homeomorphic to the closed infinite strip, then (with the given dianalytic structure) it is isomorphic to $\{z \in {\mathbb C}: -1 \leq \Im(z) \leq 1\}$.
\end{enumerate}

Let $\widehat{S}$ and $\widehat{X}$ denote the Schottky doubles of $S$ and $X$, respectively; that is, $S = \widehat{S} / \langle \tau \rangle$ and $X = \widehat{X} / \langle \mu \rangle$, where $\tau : \widehat{S} \to \widehat{S}$ and $\mu : \widehat{X} \to \widehat{X}$ are reflections.

The branched covering $p$ naturally extends to a branched covering $\widehat{p} : \widehat{S} \to \widehat{X}$. More precisely, we can view $S$ and $X$ as included in $\widehat{S}$ and $\widehat{X}$, respectively, and define
\[
    \widehat{p}(s) := \begin{cases} 
        p(s) & s \in S, \\
        \mu(p(\tau(s))) & s \notin S,
    \end{cases}
\]
whose branch locus is given by $\widehat{B} := B \sqcup \mu(B)$, where $B$ is the branch locus of $p$.

Since $X$ is compact and it has exactly two boundary components, it follows that $\widehat{X}$ is a closed surface of genus at least one if $X$ is orientable, or at least four otherwise. In both cases, the Euler characteristic satisfies $\chi(\widehat{X}) \leqslant 0$.

Observe that $\widehat{S}$ is isomorphic to a torus or to ${\mathbb C}\setminus\{0\}$, according to whether $S$ is the closed annulus or the closed infinite strip, respectively. In any of these two cases, we may consider a universal holomorphic covering map $\pi:{\mathbb C} \to \widehat{S}$. So, if $\widehat{X}$ is orientable, then $f=\widehat{p} \circ \pi:{\mathbb C} \to \widehat{X}$ is a branched holomorphic covering map whose branch locus is again $\widehat{B}$. It follows from \cite{Brody} that $\widehat{X}$ cannot be hyperbolic. If $\widehat{X}$ is non-orientable, then $\widehat{p}$ factors though the orientable double cover $f:\widehat{X}^+ \to \widehat{X}$. By the same argument as before, we obtain that $\widehat{X}^{+}$ cannot be hyperbolic, so neither can be $\widehat{X}$. In particular, we have that $\chi(\widehat{X})= 0$, so $\widehat{X}$ is either a torus or a Klein bottle. Now, as $\widehat{X}$ must have genus at least four if it is non-orientable, we have that $\widehat{X}$ cannot be the Klein bottle. As a consequence, $\widehat{X}$ is a torus and, in particular, $X$ is a closed annulus.

Next, we proceed to prove that $\widehat{p}$, so $p$, is unbranched. 

(a) Let us first consider the case when $\widehat{S}$ is the torus. In this case, as both $\widehat{X}$ and $\widehat{S}$ are both of genus one, the branched covering $\widehat{p}:\widehat{S} \to \widehat{X}$ has finite degree. 
By the Riemann–Hurwitz formula applied to $\widehat{p}$ we conclude that the branch locus $\widehat{B}=\emptyset$, and $p$ is unbranched as required.

(b) Now, let us assume $\widehat{S}$ is isomorphic to ${\mathbb C}\setminus\{0\}$. Take a simple closed curve $\gamma \subset \widehat{X}$ bounding an open disc $D$ such that all possible branch points of $\widehat{p}$ are contained in its interior. Set $A=\widehat{X} \setminus D$; which is a compact genus one surface with exactly one boundary component $\gamma$. Let $Z \subset {\mathbb C}\setminus\{0\}$ be a connected component of $\widehat{p}^{-1}(A)$, and consider the restriction $\widehat{p}|_Z:Z \to A$, which is an unbranched covering map. 

As $Z$ is planar, there should be an essential simple closed curve $\alpha$ on $A$ such that a connected component of $\widehat{p}^{-1}(\alpha)$ is a simple path that connects the two punctures $0$ and $\infty$. Let $\beta$ be another essential simple closed curve on $A$ such that intersects $\alpha$ at exactly one point. Each connected component of $\widehat{p}^{-1}(\beta)$ is either a simple closed  curve (separating $0$ and $\infty$) or it is a simple path connecting $0$ with $\infty$. The second situation cannot hold as $\alpha$ and $\beta$ are assumed to intersect at precisely one point. Since $\gamma$ is homotopic to the commutator of $\alpha$ and $\beta$, it follows that each connected component of $\widehat{p}^{-1}(\gamma)$ will be homotopic to a simple closed curve $\delta$ bounding a disc, and that $\widehat{p}|_{\delta}:\delta \to \gamma$ is a homeomorphism. Now, for each such $\delta$ let $W$ be a connected component of $\widehat{p}^{-1}(D \cup \gamma)$ containing $\delta$ in its boundary. The only possibility is that $W$ is a closed disc with $\partial W=\delta$, and $\widehat{p}|_W:W \to D \cup \gamma$ is a branched covering whose restriction to $\delta$ is a homeomorphism onto $\gamma$. This asserts that such a covering must be an homeomorphism and, in particular,  it has no branch values. Hence the branch locus $\widehat{B}=\emptyset$, and $p$ is unbranched as required.
\end{proof}

\begin{remark}
   In the case where the covering has finitely many sheets and the surface $S$ is a closed annulus, the conclusion of Lemma \ref{Lemma:BranchedCoverOfTheAnnulusStrip} can also be obtained by applying the Riemann–Hurwitz formula for branched covers between surfaces with boundary. A straightforward computation shows that the only possibility is that the surface $X$ is homeomorphic to a closed annulus and that the covering map is unbranched.
\end{remark}

%%%%%%%%%%%%%%%%%%%%%%%%%%%
%%%%%%%%%%%%%%%%%%%%%%%%%%%
\section{Proof of Theorem \ref{BirmanHildenProperty-Generalization}}\label{Sec:MainProof}
To prove that the branched covering map $p:S\rightarrow X$ with branch locus $B$ satisfies the Birman--Hilden property
we show that item (3) of Proposition \ref{Prop:Equivalence:BirmanHilden} holds: the surjective homomorphism $\Theta:\LMap(X;B) \to \SMap(S)/D$  given by $[f]\mapsto [\tilde{f}]\mod \ D$ with $\tilde{f}$  a lift of $f$, is one-to-one. Our argument is based on the strategy used to prove \cite[Theorem 1]{Winarski2015}.

Notice that $X \setminus B$ is either of infinite type or has negative Euler characteristic because the map $p: S \setminus p^{-1}(B) \to X \setminus B$ is an unramified covering and $S$ satisfies these properties. If $X \setminus B$ is a sporadic surface, whose only possibility under our hypothesis is the sphere with three punctures, then $\Map(X; B)$ is finite. In this case, Corollary \ref{Corollary:ThetaFreeTorsionFiniteTypeSurfaces} implies that the map $\Theta$ is injective, and therefore $p$ has the Birman–Hilden property. Hence, for the rest of the proof we assume that $X \setminus B$ is a non-sporadic surface.

Take a $k$-stable Alexander system of curves $\Gamma$ for $X\setminus B$. The existence of $\Gamma$ is provided by Theorem \ref{thm:AlexanderMethodInfiniteSurfaces}; furthermore $\Gamma$ can be assumed to be a $1$-stable Alexander system if $X\setminus B$ is a surface of infinite type.  

Let $f\in \LHomeo(X;B)$ be a representative of a mapping class in $\mathrm{Ker}(\Theta)$. Notice that we get the desired result if we show that:

\medskip

\noindent {\bf Claim}: \emph{For every $\gamma \in \Gamma$, the simple closed curve $f(\gamma)$ is isotopic to $\gamma$ in $X\setminus B$.}

\medskip

Indeed, in this situation, if $X\setminus B$ is of infinite type then $f$ is isotopic to the identity $id_{X\setminus B}$, and hence $f\simeq id_X\ rel\ B$. Otherwise, $X\setminus B$ is of finite type and $f^k$ is isotopic to $id_{X\setminus B}$, therefore $f$ represents a mapping class of finite order in $\LMap(X;B)$. From Corollary \ref{Corollary:ThetaFreeTorsionFiniteTypeSurfaces} it follows that $f$ must represent the trivial mapping class in $\LMap(X;B)$.

\medskip

\noindent \emph{Proof of the Claim.} Fix $\gamma \in \Gamma$ and set $\beta:=f(\gamma)$. Let us take $\widetilde{f}$ a lift of $f$; up to precomposition by a deck transformation, we can assume that $\widetilde{f}$ is isotopic to $Id_S$. In particular, we have that $\widetilde{\gamma}:=p^{-1}(\gamma)$ is isotopic to $\widetilde{f}(\widetilde{\gamma})=p^{-1}(\beta)=:\widetilde{\beta}$ in $S$. \smallskip

We divide the proof in two cases:  $\gamma\cap \beta  =\emptyset$ and $\gamma\cap \beta \neq \emptyset$.

\vspace{0.1in}

\noindent \underline{\emph{Case (1): $\gamma\cap \beta =\emptyset$}}. In this case,  $\widetilde{\gamma}$ and $\widetilde{\beta}$ are disjoint. Consider the surfaces $X^\prime:=X\setminus (\gamma \sqcup \beta)$ and $S^\prime:= S\setminus (\widetilde{\gamma}\sqcup \widetilde{\beta})$ and take $p^\prime:S^\prime\to X^\prime$ the branched covering map induced by $p$. Let $\Omega$ be the set whose elements are the closure in $S$ of the connected components of $S^\prime$. Observe that, if $W\in\Omega$, then each boundary component of $W$ is a connected component of $\widetilde{\gamma}\sqcup \widetilde{\beta}$, and the restriction $p_{\vert W}:W\to p(W)$ is a branched covering map. 

As $f$ is isotopic to the identity $Id_S$, $\widetilde{\gamma}$ is isotopic to $\widetilde{\beta}$, so we may find $W\in\Omega$ such that either $W$ is homeomorphic to a closed annulus or $W$ is homeomorphic to an closed infinite strip (if $p$ has finite degree, then $W$ is necessarily an annulus). Such $W$ has one boundary component given by a connected component of $\widetilde{\gamma}$ and one boundary component which is a connected component of $\widetilde{\beta}$. Therefore, $p(W)$ is a surface whose boundary is equal to $\gamma \sqcup \beta$. If $W$ is a closed annulus, it is clear that $p(W)$ is compact. The same is true if $W$ is the closed infinite strip, since $p$ restricted to either boundary component of $W$ is equivalent to the universal cover $\mathbb{R}\rightarrow \mathbb{S}^1$ and we can take a compact subspace $F$ of $W$ such that $p(F)=p(W)$. By Lemma \ref{Lemma:BranchedCoverOfTheAnnulusStrip}, in both cases we have that $p(W)$ is the closed annulus and $p_{\vert W}:W\to p(W)$ is unbranched. Furthermore, since $p$ is a fully-ramified branched covering map, $p(W)$ does not contain in its interior any branch point of $p$, that is $B\cap p(W)=\emptyset$. This allows us to conclude that $\gamma$ and $\beta$ are isotopic in $X\setminus B$. 

\vspace{0.1in}

\noindent \underline{\emph{Case (2): $\gamma\cap \beta \neq \emptyset$}}. We have that $\widetilde{\gamma}\cap \widetilde{\beta}\neq \emptyset$ and they are isotopic in $S$, so they form bigons in $S$. Let $\w{C}$ be an innermost bigon in $S$ bounded by $\widetilde{\gamma}$ and $\widetilde{\beta}$. By following the proof of \cite[Lemma 1]{Winarski2015}, we obtain that $C = p(\widetilde{C})$ is an innermost bigon bounded by $\gamma$ and $\beta$ in $X$ and moreover, that $C$ does not contain any branch point in its interior.

Hence, we can define an isotopy on $X \setminus B$ to eliminate the bigon $C$. This process reduces the intersection number between $\gamma$ and $\beta$. Since $\gamma$ and $\beta$ are compact, we can continue this process a finite number of times until $\gamma$ and $\beta$ are disjoint. We then apply \emph{Case (1)} to these curves, completing the proof.
\qed

%%%%%%%%%%%%%%%%%%%%%%%
\subsection{An alternative proof for unbranched regular covers}\label{Rmk:original} The original argument of Birman and Hilden in \cite[Lemma 6]{BH73} can be combined with the Alexander method to give an alternative proof of the Birman--Hilden property for unbranched regular covers over a non-compact hyperbolic base surface $X$ as we now explain.

Let $p:S \to X$ be an unbranched regular cover and suppose $X$ is noncompact and different from the sphere or the projective plane with at most two punctures. Let $p_\ast$ be the induced homomorphism on fundamental groups, $x\in X$ a base point and $\widetilde{x}\in p^{-1}(x)$. Since the fundamental group of $X$ if free of rank at least two, then the centralizer of $p_*(\pi_1(S,\widetilde{x}))$ in $\pi_1(X,x)$ is trivial.  

Now, take $\widetilde{f} \in \SHomeo(S) \cap \Homeo_0(S)$ and let $f \in \Homeo(X)$ be its projection under $p$. Following the same argument in the proof of \cite[Lemma 1.6]{BH73}, the fact that the centralizer of $p_*(\pi_1(S,\widetilde{x}))$ in $\pi_1(X,x)$ is trivial, implies (after composing with a change of base point) that $f$ induces the identity on the fundamental group. By  Alexander's method if the surface $X$ is of infinite type, and by the Dehn-Nielsen-Baer Theorem if $X$ is of finite type, we have that $f$ is isotopic to the identity. Lifting this isotopy to $S$ we obtain a fiber-preserving isotopy between $\widetilde{f}$ and a deck transformation of $p$. By Proposition \ref{Prop:DectTransfIsotopicId}, this deck transformation must be the identity $Id_S$, and therefore we conclude that $p$ satisfies the Birman--Hilden property.

\begin{remark}
A similar argument is given in \cite[Section 1.9]{LUCAS25BH} for an unbranched regular cover $p:S\to X$  with finitely many sheets between closed oriented $3$--manifolds $S$ and $X$, under the hypotheses that $X$ is aspherical, that roots are unique in $\pi_1(X)$, and that homotopic homeomorphisms of $X$ are always isotopic.
\end{remark}

Notice that if the surface $X$ is closed with negative Euler characteristic, the original argument of Birmann and Hilden cannot be adapted to include unbranched covers with infinitely many sheets. This is because they use the hypothesis that $p$ has finitely many sheets in an essential way in order to obtain that the centralizer of $p_*(\pi_1(S,\widetilde{x}))$ in $\pi_1(X,x)$ is trivial, see the proof of \cite[Theorem 1]{BH73}. Something similar happens with the argument to prove \cite[Proposition 3]{ALS09}.

%%%%%%%%%%%%%%%%%%%%%%%%
%%%%%%%%%%%%%%%%%%%%%%%%
\section{Some applications to geometrically characteristic covers}\label{Sec:Applications}
In this section, we apply the Birman–Hilden property to unbranched covers that are \emph{geometrically characteristic}  to obtain injective homomorphisms between mapping class groups and surface braid groups.

\begin{definition}\label{def:geomchar}[Geometrically characteristic cover]
Let $\Sigma$ be a surface and $x\in\Sigma$ a based point. An automorphism of $\pi_1(\Sigma,x)$ is called \emph{geometric} if it is induced by a homeomorphism of $\Sigma$ that preserves $x$. A subgroup $N\leqslant \pi_1(\Sigma,x)$ is \emph{geometrically characteristic} if $N$ is invariant under all geometric automorphisms of $\pi_1 (\Sigma,x)$. An unbranched covering map $p:S\to X$ is \emph{geometrically characteristic} if it corresponds to a geometrically characteristic subgroup of $\pi_1(X,x)$. 
\end{definition}

Examples of such type of covers include unbranched characteristic covers like homology covers, the $\mathbb{Z}/n\mathbb{Z}$-homology covers, and the orientable double cover of a non-orientable surface. The first two types have been recently studied in \cite[Theorems 1.2 $\&$ 1.3]{BCC24}; see also \cite[Section 1.1.]{BasHidalgo}, and \cite[Corollary 6.2]{AGH22}.  Furthermore, in \cite[Theorem 1.5]{BCC24} the authors classify the homeomorphism types of geometric characteristic covers, with infinitely many sheets, over a finite type orientable surface.

Notice that by definition, if $p:S\to X$ is a geometric characteristic covering map, then $\Homeo(X;x)=\LHomeo(X;x)$. Moreover, the liftable mapping class group is precisely the whole mapping class group; compare with the definition of geometric characteristic cover in \cite[Section 3.1]{ALMc2024}.  

\begin{proposition}\label{Prop:Geometrically:LMod=Mod}
    Let $p:S\to X$ be an unbranched geometrically characteristic covering map, and let $P\subset X$ be a discrete set of marked points. Then $\LMap(X;P) = \Map(X;P) $.
\end{proposition}
\begin{proof} Let $x\in X\setminus P$ be a base point of $X$. Any mapping class in $\Map(X;P)$
     can be represented by a homeomorphism $f$ in $\Homeo(X;P)$ such that $f(x)=x$. Since $p:S\to X$ is an unbranched geometrically characteristic cover, it follows that $f$ lifts to an element $\widetilde{f}\in \Homeo(S)$. Therefore, $\Map(X;P)=\LMap(X;P)$. 
\end{proof}

%%%%%%%%%%%%%%%%%%%%%%
\subsection{Injections between mapping class groups}
Let $S$ be a surface of infinite type or of finite type with negative Euler characteristic. From Corollary \ref{coro1}, we know that any unbranched covering map $p:S\to X$ satisfies the Birman-Hilden property. If $P\subset X$ is a closed and discrete set of marked points and $\widetilde{P}=p^{-1}(P)$, then $p:S\setminus \widetilde{P} \to X\setminus P$ also satisfies the Birman-Hilden property. Thus, there is a short exact sequence (see equation \eqref{Eq:ExactSequence:MCG's})
\begin{equation}\label{Eq:ShortExact:Injection}
    1\rightarrow \Deck(p) \rightarrow \SMap(S;\widetilde{P})  \xrightarrow{} \LMap(X;P) \rightarrow 1.
\end{equation}
The following result is a consequence of Proposition \ref{Prop:Geometrically:LMod=Mod} and the previous observation. It extends the result of Aramayona--Leininger--Souto \cite[Proposition 5]{ALS09} to the context of non-orientable surfaces and infinite type surfaces.

\begin{proposition}\label{Res:Injective:Homomorphism}
    Let $p \colon S \to X$ be an unbranched geometrically characteristic covering map, let $P \subset X$ be a closed and discrete set of marked points, and suppose that $S$ is a surface of infinite type or of finite type with negative Euler characteristic. If the short exact sequence in equation \eqref{Eq:ShortExact:Injection} splits, then there is an injective homomorphism
    \[
    \Map(X;P) \to \Map(S;\widetilde{P})
    \]
    obtained by lifting mapping classes.
\end{proposition}

We now apply the previous result to obtain in Corollary \ref{Cor:Injection:Characteristic}, injections of mapping class groups with a single marked point.

\begin{proof}[Proof of Corollary \ref{Cor:Injection:Characteristic}]
    Since the cover is geometrically characteristic, we obtain a homomorphism
    \[
        \Homeo(X; \{x\}) \to \Homeo(S; p^{-1}(x)), \qquad f \mapsto \widetilde{f}; 
    \]
    where $\widetilde{f}$ is the unique lift satisfying $\widetilde{f}(\widetilde{x}) = \widetilde{x}$. This defines a homomorphism $\Map(X; \{x\}) \to \SMap(S; p^{-1}(x))$ that splits the short exact sequence \eqref{Eq:ShortExact:Injection}. Thus, the result follows from Proposition \ref{Res:Injective:Homomorphism}.
\end{proof}

%%%%%%%%%%%%%%%%%%%%
\subsection{Mapping class group of non-orientable surfaces}\label{sec:orientable double cover}
Consider the orientable double cover $p \colon S \to N$ of a connected non-orientable hyperbolic surface $N$. Let us recall that the orientable double cover of a non-orientable surface with empty boundary coincides with its complex double cover, see Subsection \ref{SubSec:KleinSurfaces}. It is well known that $p$ is unbranched and characteristic, and moreover any homeomorphism $f\colon N\to N$ admits a unique lift $\widetilde{f} \colon S \to S$ that preserves the orientation (in fact, it admits exactly two liftings, one that preserves the orientation and another that reverses it). This yields a well-defined homomorphism
\begin{equation}\label{Eq:Injection:Orientable:Double}
    \phi \colon \Map(N) \to \SMap(S), \qquad [f] \mapsto [\widetilde{f}]; 
\end{equation}	
which splits the short exact sequence \eqref{Eq:ShortExact:Injection}. By Proposition \ref{Res:Injective:Homomorphism}, we obtain that the homomorphism $\phi$ is injective. This generalizes results of Birman and Chillingworth \cite{BC72} and Lima-Gon{\c{c}}alves, Guaschi, and Maldonado \cite[Theorem 1.1]{GGMaldonado2018NonOrientable} to the context of infinite type surfaces.

\begin{corollary}[Mapping class groups of a non-orientable surface and its orientable double cover]\label{Corollary:DoubleCover}
Let $p:S\to N$ be the orientable double cover of a non-orientable surface $N$ with empty boundary, and suppose that $S$ is of infinite type or of finite type with negative Euler characteristic. Then the induced map 
\[
\phi: \Map(N)\to \Map(S), \ \ \ \ \ \ [f]\mapsto [\widetilde{f}];
\]
is injective, where $\widetilde{f}$ is the unique lift of $f$ that preserves the orientation of $S$.    
\end{corollary} 

\begin{remark}
In \cite[Remark 2.4.]{GGMaldonado2018NonOrientable} it is shown that the induced homomorphism $\Map(N_2)\to \Map(S_1)$ by the orientable double cover $S_1\to N_2$ of the Klein bottle is not injective. The remaining case not covered by Corollary \ref{Corollary:DoubleCover} is the orientable double cover $S_0^2\to N_1^1$ of the projective plane with one marked point. In this situation, the natural induced homomorphism $\Map(S_0^2)\to \Map(N_1^1)$ is an isomorphism. 
\end{remark}

%%%%%%%%%%%%%%%%%%%%
\subsection{Injections between surface braid groups} 
For $k\geq 1$, let $\operatorname{UConf}_k(\Sigma)$ denote the \textit{unordered configuration space} of $k$ particles on the surface $\Sigma$, and take $P\in \operatorname{UConf}_k(\Sigma)$.

Recall that the \emph{surface braid group on $k$ strands} is defined as  
\[
    B_k(\Sigma) = \pi_1(\operatorname{UConf}_k(\Sigma),P).
\]
 
The evaluation map on the unordered configuration $P$ is a fiber bundle 

\begin{eqnarray*}
    \mathcal{E}_P:\Homeo(\Sigma) & \to & \operatorname{UConf}_k(\Sigma)\\
    f & \mapsto & f(P),
\end{eqnarray*}

whose fiber over $P$ is $\Homeo(\Sigma; P)$, see \cite[Lemma 1.35]{KT08}. 

Consider the orientable double cover $p: S \to N$ of a non-orientable surface $N$. Fix $P \in \operatorname{UConf}_k(N)$ and let  $\widetilde{P}=p^{-1}(P)\in\operatorname{UConf}_{2k}(S)$.  The corresponding evaluation maps $\mathcal{E}_P$ and $\mathcal{E}_{\widetilde{P}}$ give the following commutative diagram:
\[
    \xymatrix{
        \Homeo(N;P) \ar[r] \ar[d]_{\widetilde{\phi}_k} &  \Homeo(N) \ar[r]^{\mathcal{E}_P} \ar[d]_-{\widetilde{\phi}} & \operatorname{UConf}_k(N) \ar[d]^-{\widetilde{\rho}} \\
        \Homeo(S;\widetilde{P}) \ar[r] & \Homeo(S) \ar[r]^{\mathcal{E}_{\widetilde{P}}}  & \operatorname{UConf}_{2k}(S). \\
    }
\]
Here the map $\widetilde{\rho}$ sends each configuration $\{ x_1,\ldots,x_k \} \in \operatorname{UConf}_k(N)$ to its preimage under $p$, that is, 
\[
    \widetilde{\rho}(\{ x_1,\ldots,x_k \} ) = p^{-1}(\{ x_1,\ldots,x_k\}) = \{ \widetilde{x}_1,\widetilde{x}_2, \ldots,\widetilde{x}_{2k-1}, \widetilde{x}_{2k} \}.
\]
The maps $\widetilde{\phi}$ and $\widetilde{\phi}_k$ assign to each homeomorphism $f$ of $N$ its unique orientation-preserving lift $\widetilde{f}$ on $S$. If $N$ is of infinite type or of finite type type negative Euler characteristic, then, by \cite[Theorem 5.1 and 5.3]{Ham66} and \cite[Theorem 1.1 (ii)]{Yag00H}, both $\Homeo_0(N)$ and $\Homeo_0(S)$ are contractible. Hence, the fiber bundles above induce the following commutative diagram between \textit{Birman exact sequences}:
\[
    \xymatrix{
        1 \ar[r]  & B_k(N) \ar[r] \ar[d]_-{\rho} & \Map(N; P) \ar[r] \ar[d]^-{\phi_k}  & \Map(N) \ar[r] \ar[d]^-{\phi} & 1  \\
        1 \ar[r]  & B_{2k}(S) \ar[r] & \Map(S; \widetilde{P}) \ar[r]   & \Map(S) \ar[r] & 1.
    }
\]
By Corollary \ref{Corollary:DoubleCover}, the homomorphisms $\phi$ and $\phi_k$ are injective. It follows that $\rho$ is also injective. We thus obtain the following result which generalizes \cite[Corollary 2]{GG12BraidsCovers} of Gon\c{c}alves and Guaschi. In particular, our argument above gives an alternative proof of \cite[Corollary 2]{GG12BraidsCovers} in the case when the surface $N$ is a non-orientable finite type surface, with empty boundary and negative Euler characteristic.

\begin{corollary}[Braid groups of a non-orientable surface and its orientable double cover] \label{Cor:Injectivity:Braids}
Let $p : S \to N$ be the orientable double cover of a non-orientable surface $N$, with empty boundary, of infinite type or of finite type with negative Euler characteristic. Then the induced homomorphism  
\[
    \rho: B_k(N) \to B_{2k}(S),
\]
on surface braid groups is injective. 
\end{corollary}

In \cite[Theorem 1]{GG12BraidsCovers}, Gon\c{c}alves and Guaschi prove a more general result in the finite type setting: given a finite type surface $X$, possibly with non-empty boundary, and a covering $p : S \to X$ of degree $d$,  the induced homomorphism between surface braid groups $B_k(X) \to B_{dk}(S)$ is injective for every $k \in \mathbb{N}$. Their proof of this result is by induction on the number of strands $k$, using fibrations that forget a point in configuration spaces and the associated long exact sequences in homotopy. 

In contrast, our argument depends on the injectivity of the homomorphism between mapping class groups $\Map(N;P)\to \Map(S;\widetilde{P})$, but applies only to the case of the orientable double cover. However, our method extends to any unbranched geometrically characteristic cover $p:S\to X$ of degree $d$ via Corollary \ref{Cor:Injection:Characteristic}, which yields injective homomorphisms $\Map(X^*)\to \Map(S^*)$ and $\Map(X^*;P)\to \Map(S^*;\widetilde{P})$, where $X^* = X \setminus \{x\}$ and $S^* = S \setminus p^{-1}(x)$ for some point $x \in X$ and a finite set $P \subset X^*$. By an analogous argument, we obtain the following result.

\begin{corollary}[Injections between surface braid groups]\label{Cor:Braids2} Let $p \colon S \to X$ be an unbranched  geometric characteristic covering map of degree $d$ between connected topological surfaces  without boundary $S$ and $X$. Suppose that $S$ is of infinite type or of finite type with negative Euler characteristic. Then the induced homomorphism
    \[
        B_k(X^*) \to B_{dk}(S^*),
    \]
    on surface braid groups is injective, where $X^* = X \setminus \{x\}$ and $S^* = S \setminus p^{-1}(x)$ for some point $x \in X$.
\end{corollary}

%%%%%%%%%%%%%%%%%%%
%%%%%%%%%%%%%%%%%%
\subsection*{Acknowledgements}
The authors thank Porfirio L. León Álvarez for his question about injectivity between braid groups, which was answered in the affirmative in our Corollary \ref{Cor:OrientableDoubleCover}. N. Colin was funded by SECIHTI through the program \textit{Estancias Posdoctorales por México}. R. A. Hidalgo was supported by ANID Fondecyt grant 1230001. I. Morales was supported by ANID Fondecyt Postdoctoral grant 3240229. S. Quispe was supported by ANID Fondecyt grant 1220261.

\bibliographystyle{alpha}
\bibliography{References}
\end{document}